\documentclass[a4paper, 12pt]{article}

\usepackage{epsfig,color}
\usepackage{amssymb,latexsym}
\usepackage{graphicx}%
\usepackage{amsmath,graphics,latexsym,amsfonts,indentfirst}
\usepackage{mathrsfs, epstopdf,cite} 
\usepackage{pinlabel}
\newcommand {\red} {\textcolor{red}}

\newtheorem{theo}{Theorem}[section] 

\newtheorem{lemm}[theo]{Lemma}

\newtheorem{prop}[theo]{Proposition}
\newtheorem{prob}{Problem}[section]

\newcommand \hyh { \mathcal{H} }
\newcommand \hyc { \mathcal{C} }
\newcommand \hye { \mathcal{E} }
\newcommand \hyd { \mathcal{D} }

\newcommand \hyl { \mathcal{L} }

\renewcommand\red {}

\newcommand{\up}{\textup}

\begin{document}

\title{Zero-free intervals of chromatic polynomials of hypergraphs
}
\author{Ruixue Zhang and 
Fengming Dong\thanks{Corresponding author.
Email: fengming.dong@nie.edu.sg}
\\
\small
Mathematics and Mathematics Education, National
Institute of Education
\\
\small Nanyang Technological University, Singapore
}
\date{}

\maketitle

\begin{abstract} 
In this paper, we prove that 
\red{$(-\infty, 0)$ 
is a zero-free interval 
for chromatic polynomials of 
a family $\hyl_0$ of hypergraphs
and $(0, 1)$ is a zero-free interval for chromatic polynomials of a subfamily $\hyl_0'$ of $\hyl_0$ of hypergraphs.}
These results extend known results on zero-free intervals of 
chromatic polynomials of graphs and hypergraphs.
\end{abstract}



\section{Introduction}

For any simple graph $G=(V, E)$ and any positive integer $k$, 
a \textit{proper $k$-coloring of $G$} is a mapping 
\red{$\phi: V\rightarrow \{1, 2, \ldots,k\}$}
such that $\phi(u)\neq \phi(v)$ holds for each pair of adjacent vertices 
$u$ and $v$ in $G$.
The \textit{chromatic polynomial} $P(G, \lambda)$ of $G$
is the function which counts 
the number of proper $\lambda$-colorings of $G$
whenever $\lambda$ is a positive integer.
It was introduced by Birkhoff \cite{birkhoff1946chromatic} in 1912
for planar graphs, in hope of proving the four-color conjecture,
and extended to all graphs 
by Whitney \cite{whitney1932logical} in 1932.
Although Birkhoff's \red{attempt} failed, the study of 
chromatic polynomials is one of the active research areas
\red{in graph theory}, especially the study of 
the zero distribution of chromatic polynomials~
\red{\cite{
jackson1993zero, jackson2003zeros, royle2009recent,
thomassen1997zero,woodall1977zeros,woodall1997largest}.}

It was proved by Sokal~\cite{sokal2004chromatic} in 2004 that 
the zeros of chromatic polynomials of graphs 
are dense in the whole 
complex plane. 
It is also known that $(-\infty, 0), (0, 1)$ and $(1, 32/27]$ are the only three maximal zero-free \red{real intervals} for 
chromatic polynomials of graphs
(see 
\red{\cite{jackson1993zero,thomassen1997zero,tutte1974chromials}}
), 
where the zero-free interval $(1,32/27]$ was found by Jackson~\cite{jackson1993zero}
and was proved to be the last such interval 
by Thomassen~\cite{thomassen1997zero}.
For certain subsets of graphs, 
Thomassen \cite{thomassen2000chromatic} showed that 
if $G$ is any graph with a Hamiltonian path, 
then $P(G,\lambda)$ is zero-free in the interval 
$(1,1.2955\cdots)$, while 
Dong and Jackson \cite{dong2011zero} proved that 
for any 3-connected plannar graph $G$, 
$P(G,\lambda)$ is zero-free in the interval 
$(1,1.2040\cdots) $.

The study of chromatic polynomials has been extended to
chromatic polynomials of hypergraphs for \red{many years}.
A \textit{hypergraph} $\hyh$
\red{is an \red{ordered} pair $(V(\hyh), \hye(\hyh))$,
where $V(\hyh)$ is a finite set,
called the \textit{vertex set} of $\hyh$,
and $\hye(\hyh)$ is a set of non-empty subsets of 
$V(\hyh)$, called the \textit{edge set} of $\hyh$.
Obviously, $\hye (\hyh)\subseteq \{e\subseteq V(\hyh): |e|\geq 1\}$. }
Thus a graph is a hypergraph $\hyh$ with
$|e|\leq 2$ for each $e\in \hye(\hyh)$.
Regarding \red{colorings} of hypergraphs, 
there are several different definitions,
such as strong proper colorings \cite{agnarsson2004strong} and 
$\hyc$-colorings \cite{voloshin1993mixed}.
In this paper, we work with weak proper colorings of hypergraphs which were
introduced by \red{Erd\H{o}s} and Hajnal \cite{erdHos1966chromatic} in 1966.
For any integer $k\geq 1$, 
a \textit{weak proper $k$-coloring} of a hypergraph $\hyh=(V, \hye)$ is 
\red{an assignment of colors from 
$\{1,\ldots, k\}$
to the vertices so that 
each edge contains at least two vertices with different colors,
i.e., there
is a mapping $\phi: V\rightarrow \{1,\ldots, k\}$} such that 
$|\{\phi(v): v\in e\}|\geq 2$
holds for each $e\in \hye$.
The {\it chromatic polynomial} of $\hyh$,
denoted by $P(\hyh, \lambda)$,
is the function which counts the number of 
weak proper $k$-colorings of $\hyh$ whenever $\lambda=k$ 
is a positive integer,
\red{and it is indeed a polynomial in $\lambda$ 
\cite{dohmen1995broken,tomescu1998chromatic}.}

\red{A weak proper $k$-coloring of
a hypergraph $\hyh$ is actually a proper $k$-coloring of a graph whenever $\hyh$ is a graph.
In such case, $P(\hyh, \lambda)$ is a chromatic polynomial of a graph.
Thus chromatic polynomials 
of hypergraphs are generalizations of  chromatic polynomials of graphs.}
Many properties \red{of} chromatic polynomials of graphs have been 
extended to chromatic polynomials of hypergraphs 
\red{(for example, see \cite{allagan2007chromatic, allagan2014chromaticwheel, allagan2014chromatic,
borowiecki2000chromatic, dohmen1995broken,tomescu1998chromatic, tomescu2004sunflower, tomescu2007chromaticity,
 tomescu2009some, tomescu2014hypergraphs, 
voloshin1995upper, voloshin2002coloring, walter2009some}).}
Meanwhile in \red{\cite{zhang2017properties}}, the authors of this article
also showed that some properties \red{of} chromatic 
polynomials of hypergraphs do not hold 
for chromatic polynomials of 
graphs.
One of these properties is that
chromatic polynomials of hypergraphs have 
\red{real zeros that are dense}
in the 
whole set of real numbers,
while chromatic polynomials of graphs 
have three zero-free intervals.
\red{It is natural for people to look for a suitable way of extending 
the result that $(-\infty,0)$ and $(0,1)$ 
are zero-free intervals of chromatic polynomials of graphs
to chromatic polynomials of hypergraphs.}

A \textit{cycle} $C$ in $\hyh$ is defined to be a sequence
of alternating vertices and edges:
$(v_{1}, e_{1}, v_{2}, e_{2}, \ldots, v_{t}, e_{t}, v_1),$
where $t\ge 2$, 
$v_1,\ldots,v_t$ are 
pairwise distinct vertices
and $e_1, \ldots, e_t$ are pairwise distinct edges, such that 
$\{v_{i}, v_{i+1}\}\subseteq e_{i}$ for $i\in \{1, \ldots, t\}$,
where $v_{t+1}=v_{1}$.
Sometimes, a cycle $C$ is also written as $(e_1,\ldots, e_t)$ for short.
In this case, we assume that there exist pairwise distinct 
vertices $v_1, \ldots,v_t$
such that $v_{i+1}\in e_i\cap e_{i+1}$ for all $i=0,1,\ldots,t-1$,
where $e_{0}=e_t$. 

Let \red{$\hyl_1$ be the set of 
hypergraphs 
in which each edge has an even size and 
each cycle $C=(e_1,\ldots,e_r)$
contains $e_i$ with $|e_i|=2$ for some $i \in \{1, \ldots, r\}$.}
Dohmen \cite{dohmen1995broken} proved that 
for any hypergraph $\hyh$ in $\hyl_1$,
$P(\hyh,\lambda)\ne 0$ holds
for any real
$\lambda\in (-\infty, 0)$.

\begin{theo}[\cite{dohmen1995broken}]
\label{doh-th}
\red{If $\hyh=(V, \hye)$ is a hypergraph in $\hyl_1$, 
then }
$(-1)^{|V|}P(\mathcal{H}, \lambda)>0$
holds for all real $\lambda\in (-\infty, 0)$.
\end{theo}

Theorem~\ref{doh-th} certainly includes the fact that
\red{chromatic polynomials of
graphs has no real zero in the interval 
$(-\infty, 0)$.} 
As we know, this result is currently 
the only known result on zero-free intervals for chromatic polynomials of hypergraphs.

Observe that the hypergraph $\hyh$ in Figure \ref{Fig0}
does not belong to $\hyl_1$, as 
\red{there is a cycle $(e_1, e_2, e_3)$ in $\hyh$
with $|e_i|>2$ for all $i=1,2,3$.}
It can be calculated that 
\[
P(\hyh, \lambda)=\lambda^6(\lambda-1)^4(\lambda^3-2)
(\lambda^2+\lambda+1)^2,
\]
\red{which shows that although $\hyh$ is not in $\hyl_1$, 
its chromatic polynomial $P(\hyh, \lambda)$ 
is zero-free in the interval $(-\infty, 0)$.}

In this paper, we 
\red{first extend} Dohmen's result to a larger 
set of hypergraphs,
\red{which definitely includes the hypergraph
shown in Figure \ref{Fig0}. 
Futher we prove} the existence 
of an infinite set of hypergraphs \red{with at least one edge 
of size greater than $2$} whose chromatic polynomials
are zero-free in \red{the interval} $(0,1)$.

\vspace*{0.5cm}
\begin{figure}[htp]
\labellist
\pinlabel $e_6$ at 200 420
\pinlabel $e_1$ at 350 380
\pinlabel $e_2$ at 350 250
\pinlabel $e_4$ at 200 200
\pinlabel $e_3$ at 230 315
\pinlabel $e_5$ at 85 310
\pinlabel $e_7$ at 340 315
\pinlabel $e_8$ at 180 305
\endlabellist
\centering
\includegraphics[scale=0.5]{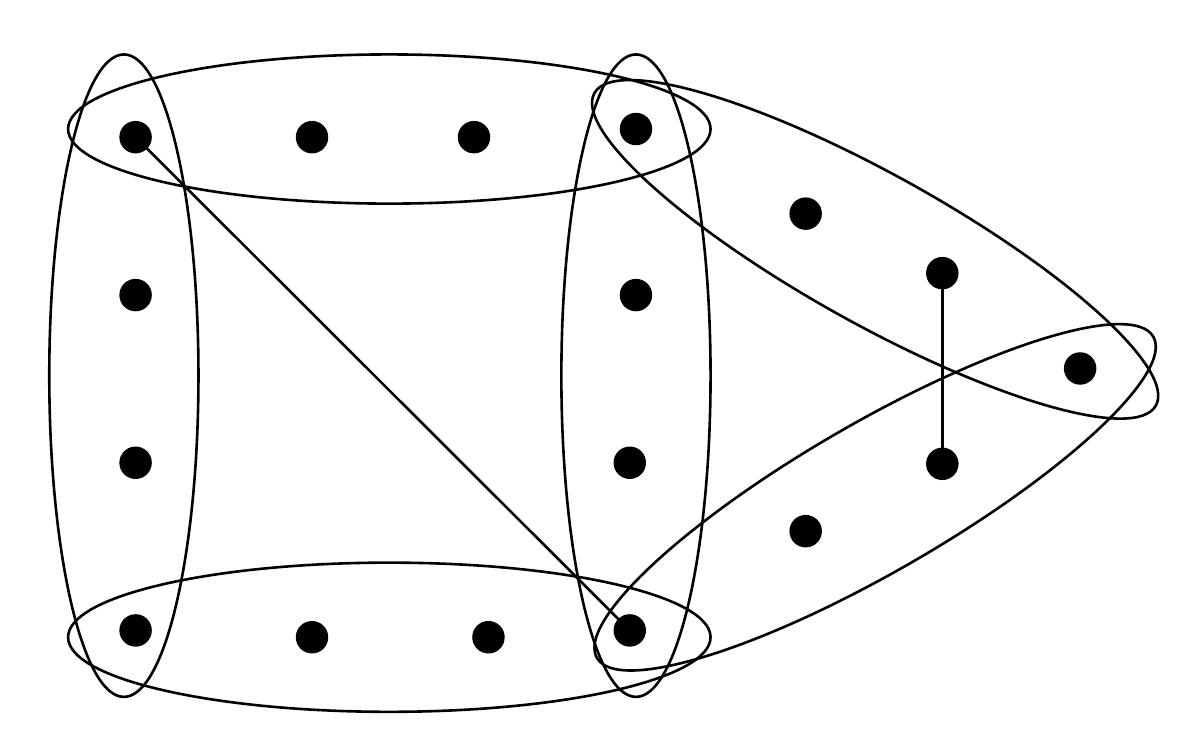}
\caption{A hypergraph in $\hyl_0\setminus \hyl_1$}\label{Fig0}
\end{figure}

Let $\hyl_0$ be the set of hypergraphs $\hyh$
\red{in which each edge has an even size and each cycle 
$C=(e_1,\ldots,e_r)$ contains
distinct vertices $u$ and $v$ in the set 
$\bigcup_{1\le i\le r}e_i$ 
forming an edge in $\hyh$.
For instance, the hypergraph $\hyh$
shown in Figure \ref{Fig0} is in $\hyl_0\setminus \hyl_1$,
which also implies that $\hyl_1$ is a proper subset of $\hyl_0$.
Actually there are infinitely many hypergraphs 
in $\hyl_0\setminus \hyl_1$.
Let $\hyl_0'$ be the set of hypergraphs 
$\hyh=(V, \hye)\in\hyl_0$ containing a connected 
spanning subhypergraph $(V,\hye_2)$, where $\hye_2=\{e\in \hye: |e|=2\}$.
Clearly, $\hyl_0'\subset \hyl_0$.}

We prove the following two results 
on the chromatic polynomials 
of hypergraphs in $\hyl_0$ and $\hyl'_0$ in this article. 

\begin{theo}\label{Th-main1}
\red{If $\mathcal{H}=(V, \hye)$ is a hypergraph in $\hyl_0$, 
then}
$(-1)^{|V|}P(\mathcal{H}, \lambda)>0$
holds for all real $\lambda\in (-\infty, 0)$.
\end{theo}

\red{
Theorem~\ref{Th-main1} implies  
Theorem~\ref{doh-th} directly,
as $\hyl_1\subset \hyl_0$.}

\begin{theo}\label{Th-main2}
\red{If $\mathcal{H}=(V, \hye)$ is a hypergraph in $\hyl_0'$, then} 
\begin{enumerate}
\item[(i)] \red{$(-1)^{|V|+1}P(\mathcal{H}, \lambda)>0$}
holds for all real $\lambda\in (0,1)$; and 
\item[(ii)] $P(\mathcal{H}, \lambda)$ has no multiple 
\red{zero} at $\lambda=0$.
\end{enumerate}
\end{theo}

\red{Note that Theorem \ref{Th-main2}
fails for some hypergraphs in $\hyl_0\setminus \hyl_0'$. 
For instance,
the hypergraph $\hyh$ shown in Figure \ref{nFig2} 
belongs to  $\hyl_0\setminus \hyl_0'$ 
and its chromatic polynomial is
\[
P(\hyh, \lambda)=\lambda(\lambda-1)(\lambda-2)(\lambda(\lambda-2)^2-1),
\]
which has a real zero around 0.38,
implying that 
Theorem \ref{Th-main2} fails for this hypergraph.}

\begin{figure}[htp]
\centering
\includegraphics[scale=0.6]{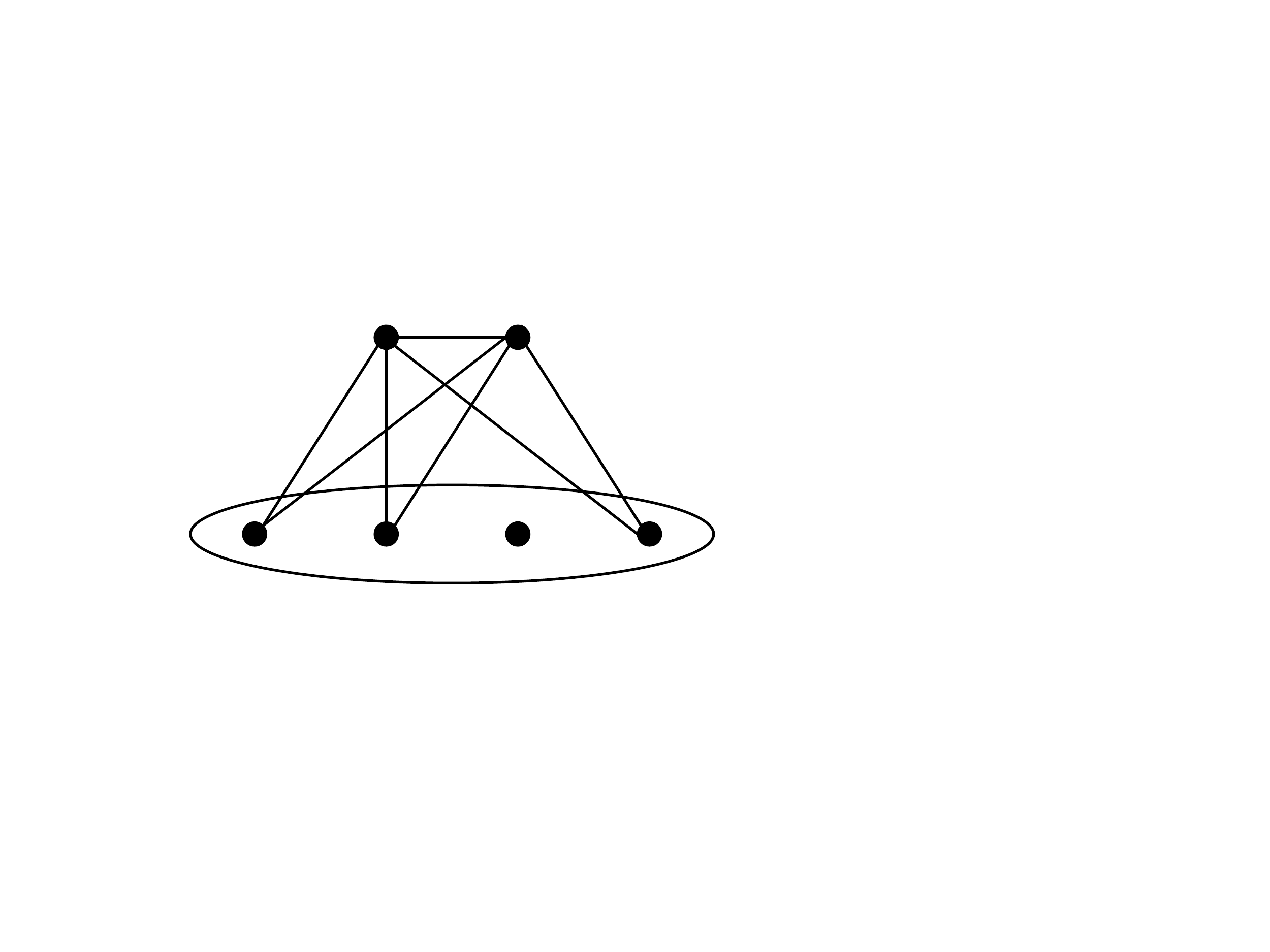}
\caption{A hypergraph in $\hyl_0\setminus \hyl_0'$}
\label{nFig2}
\end{figure}

\red{The proofs of Theorems~\ref{Th-main1} 
and~\ref{Th-main2} are given in Section~\ref{S4}.}

\section{Preliminaries}\label{S2}

In this section, 
we present several known results on 
chromatic polynomials of hypergraphs, 
which will be applied later.
The first one follows directly from the definition of weak proper colorings of a hypergraph.

\begin{prop}\up{\cite{jones1976some}}\label{PP2.1}
Let \red{$e_1$ and $e_2$ be} two edges in a hypergraph $\hyh$. 
If $e_1\subseteq e_2$, then 
$$P(\hyh, \lambda)=P(\hyh-e_2, \lambda),$$
where $\hyh-e_2$ is the hypergraph obtained from $\hyh$ by removing $e_2$.
\end{prop}

Proposition \ref{PP2.1} shows that 
we need only to consider Sperner hypergraphs when studying chromatic 
polynomials of hypergraphs,
where a hypergraph $\hyh=(V,\hye)$ is called {\it Sperner}
if $e_1\nsubseteq e_2$ holds for 
any distinct edges $e_1, e_2\in \hye$.

For any hypergraph $\hyh=(V,\hye)$, 
its \textit{Sperner subhypergraph} 
is defined to be $\hyh^*=(V, \hye^*)$, 
where $\hye^*$ is the minimal subset of $\hye$ such that for each $e\in \hye\setminus \hye^*$,
there exists $e'\in \hye^*$ with $e'\subseteq e$.
That means, $\hyh^*$ can be obtained from $\hyh$ by removing any edge 
$e\in \hye$ whenever $e'\subseteq e$ holds for another edge $e'\in \hye$.        
It is not difficult to show that $\hye^*$ is uniquely determined by $\hye$.
Obviously, $\hyh^*$ is Sperner,  and 
$\hyh^*=\hyh$ when $\hyh$ is Sperner.
Moreover, by Proposition \ref{PP2.1},
\begin{equation}\label{nEq1}
P(\hyh, \lambda)=P(\hyh^*, \lambda).
\end{equation}

\red{The following proposition can be obtained from the definition of Sperner hypergaphs directly.}

\begin{prop}\label{fact}
\red{If $\hyh=(V, \hye)$ is Sperner and $e_0$ is an edge in $\hyh$
with $|e_0|=2$,
then $|e_0\cap e|\leq 1$ holds for any edge $e\in \hye\setminus \{e_0\}$.}
\end{prop}

Let $\hyh=(V, \hye)$ be a hypergraph and $V_0\subset V$. 
Let $\hyh\cdot V_0$ denote the hypergraph obtained from $\hyh$ by identifying all vertices in $V_0$ as one, 
i.e., $\hyh\cdot V_0$ is the hypergraph with vertex set $(V-V_0)\cup \{w\}$ 
and edge set 
$$
\{e\in \hye: e\cap V_0=\emptyset\}\cup \{(e-V_0)\cup \{w\}: e\in \hye, e\cap V_0\neq \emptyset\},
$$
\red{where $w$ is the new vertex produced
when $\hyh\cdot V_0$ is obtained from 
$\hyh$ by identifying all vertices in $V_0$.
For an edge $e$ in $\hyh$, 
let $\hyh/e$ denote the hypergraph $(\hyh-e)\cdot e$.
We also say that
$\hyh/e$ is obtained from $\hyh$ by \textit{contracting} the edge $e$.}

The deletion-contraction formula for chromatic polynomials of graphs is very important for the computation of this polynomial \cite{biggs1993algebraic,
birkhoff1946chromatic, fengming2005chromatic, read1968introduction, 
read1988chromatic}.
It was extended to chromatic polynomials of hypergraphs by Jones \cite{jones1976colourings}.

\begin{theo}\up{\cite{jones1976colourings}}\label{TH2.3}
Let $\hyh=(V, \hye)$ be a hypergraph.
For any $e\in \hye$, 
\begin{equation}
P(\hyh, \lambda)=P(\hyh-e, \lambda)-P(\hyh/e, \lambda).
\end{equation}
\end{theo}


\section{$\hyl_0$ is \red{deletion and contraction closed}}

In this section, we show that 
for any Sperner hypergraph $\hyh\in \hyl_0$ with at least one edge,
there is an edge $e_0$ in $\hyh$ 
such that both $\hyh-e_0$ and $(\hyh/e_0)^*$
belong to $\hyl_0$.

For any hypergraph $\hyh=(V, \hye)$ and $e\in \hye$,
let $\hyh/e$ be written as $(V/e, \hye/e)$.
By the notation for the Sperner subhypergraph,
$(\hyh/e)^*$ is written as $(V/e, (\hye/e)^*)$.

\begin{lemm}\label{LE2.5}
Let $\hyh=(V, \hye)\in \hyl_0$ 
and $e_0\in \hye$. Assume that $\hyh$ is Sperner.
Then, 
\begin{equation}\label{Eq5}
(\hye/e_0)^*\subseteq \{e\in \hye: e\cap e_0=\emptyset \}\cup 
\{(e-e_0)\cup \{w\}: |e\cap e_0|=1, e\in \hye\},
\end{equation}
where $w\notin V$.
\end{lemm}

\begin{proof}
By the definition of $\hyh/e_0$, we have
\begin{equation}\label{Eq6}
\hye/e_0=\{e\in \hye: e\cap e_0=\emptyset\}\cup \{(e-e_0)\cup \{w\}: e\in \hye, e\cap e_0\neq \emptyset\}.
\end{equation}

By the definition of $(\hyh/e_0)^*$ and (\ref{Eq6}), 
to prove (\ref{Eq5}),
it suffices to show that for any $e_1\in \hye$ with $|e_1\cap e_0|\geq 2$, 
there exists $e_2\in \hye$ with $|e_2\cap e_0|=1$
such that $e_2-e_0\subseteq e_1-e_0$.

\red{
Let $e_1\in \hye$ with $\{u, v\}\subseteq e_1\cap e_0$, 
as shown in Figure \ref{Fig1}.
Observe that $(u, e_0, v, e_1, u)$ forms a cycle in $\hyh$.
Since $\hyh\in \hyl_0$, 
there exists an edge $e_2=\{x, y\}\in \hye$ such that $e_2\subseteq e_0\cup e_1$.
Moreover, $e_2\not\subseteq e_i$ 
for $i\in \{0,1\}$ as $\hyh$ is Sperner.
Without loss of generality,
we assume
 that $x\in e_1-e_0$ and $y\in e_0-e_1$,
as shown in Figure \ref{Fig1}.
Clearly, $e_2-e_0=\{x\}\subseteq e_1-e_0$, as required. 
Thus the result holds.}
\end{proof}
\begin{figure}[htp]
\labellist
\pinlabel {$e_0$} at 180 440
\pinlabel $e_1$ at 180 290
\pinlabel $u$ at 155 360
\pinlabel $v$ at 220 360
\pinlabel {$x$} at 220 335
\pinlabel {$y$} at 220 400
\pinlabel {$e_2$} at 250 350
\endlabellist
\centering
\includegraphics[scale=0.65]{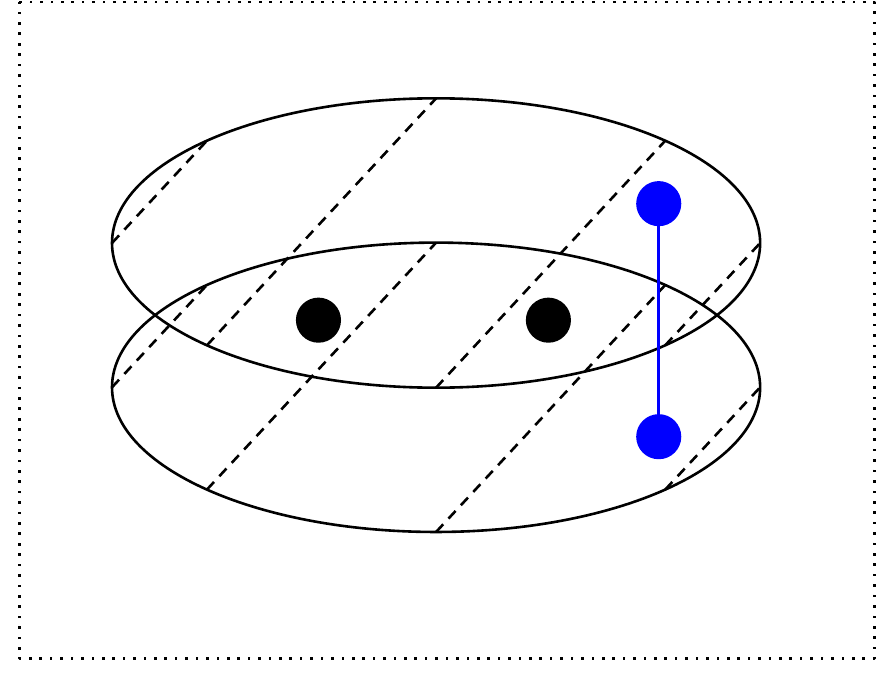}\\
\caption{A hypergraph $\hyh\in \hyl_0$ with $|e_1\cap e_0|\geq 2$}\label{Fig1}
\end{figure}


If $\hyh\in \hyl_0$ is a graph, then it is
obvious that $\hyh-e_0\in \hyl_0$ and $(\hyh/e_0)^*\in \hyl_0$.
In the following, we consider the case that
$\hyh$ is not a graph, i.e., there exists at least one edge
$e\in \hye$ with $|e|>2$. 

\begin{lemm}\label{LE2.6}
Let $\hyh=(V, \hye)\in \hyl_0$ 
and $e_0\in\hye$ with $|e_0|>2$. 
Assume that $\hyh$ is Sperner.
Then, 
$\hyh-e_0\in \hyl_0$ and $(\hyh/e_0)^*\in \hyl_0$.
\end{lemm}

\begin{proof}
It is obvious that $\hyh-e_0\in \hyl_0$
as it is obtained from $\hyh$ by removing $e_0$ directly.

\red{Let $(\hyh/e_0)^*$ be the hypergraph $(V/e_0, (\hye/e_0)^*)$.
For any vertex $u$ in $(\hyh/e_0)^*$, 
by the definition of contracting,
either $u\in V-e_0$ or $u=w$, 
where $w$ is the new vertex in 
$\hyh/e_0$ after contracting $e_0$
in $\hyh$.}
For any edge $e^*\in (\hye/e_0)^*$,
by Lemma \ref{LE2.5}, 
either $e^*\in \{e\in \hye: e\cap e_0=\emptyset\}$ or 
$e^*\in \{(e-e_0)\cup \{w\}: e\in \hye, |e\cap e_0|=1\}$,
implying that $e^*$ always has the even size as 
$\hyh\in \hyl_0$. 
Thus, to show that $(\hyh/e_0)^*\in \hyl_0$,
it remains proving the following claim.

\noindent {\bf Claim}:
For any cycle $C^*=(u_1, e_1^*, u_2, e_2^*, 
\ldots, u_t, e_t^*, u_1)$ in $(\hyh/e_0)^*$,
there exist $u,v\in \bigcup_{1\le i\le t}e^*_i$ 
such that $\{u,v\}\in (\hye/e_0)^*$.

\red{To prove this claim,
it suffices to show that 
the claim holds for each minimal cycle $C^*$ in $(\hyh/e_0)^*$,
 i.e., a cycle that does not contain any other cycle
in $(\hyh/e_0)^*$.
In the following, 
we assume that $C^*$ is a minimal cycle in $(\hyh/e_0)^*$.
It is trivial if $|e_i^*|=2$ for some $i\in \{1, \ldots, t\}$.
Thus we also assume that $|e_i^*|>2$ for all $i=1, \ldots, t$.}


By Lemma~\ref{LE2.5},
there exist edges $e_1, \ldots, e_t$ in $\hyh$ with the following properties:  for all $i=1, \ldots, t$,
\begin{enumerate}
\item[(a)] $|e_i\cap e_0|\leq 1$; 
\item[(b)] $e_i^*=e_i$ when $e_i\cap e_0=\emptyset$; and 
$e_i^*=(e_i-e_0)\cup \{w\}$ otherwise. 
\end{enumerate}

\red{Let $P_0=e_0\cap \bigcup_{1\leq i\leq t}e_i$.
Clearly, 
$\bigcup_{1\leq i\leq t}e_i^*=\bigcup_{1\leq i\leq t}e_i$
when $P_0=\emptyset$, 
and $\bigcup_{1\leq i\leq t}e_i^*=(\bigcup_{1\leq i\leq t}e_i-P_0)\cup \{w\}$ otherwise.}

\red{
We are now going to prove the claim on a case-by-case basis.
}

\red{
\noindent
\textbf{Case 1:} $w\notin \{u_1, \ldots, u_t\}$.}

\red{
In this case, 
either 
$w\notin \bigcup_{1\leq i\leq t}e_i^*$ 
or 
$w\in \bigcup_{1\leq i\leq t}e_i^*$ but $w\notin \{
u_1, \ldots, u_t\}$.
Thus there is a cycle 
$C=(u_1, e_1, u_2, e_2, \ldots, u_t, e_t, u_1)$ in $\hyh$.
Since $\hyh\in \hyl_0$,
there is an edge $e'=\{u, v\}\in \hye$ such that $e'\subseteq \bigcup_{1\leq i\leq t}e_i$.
By Proposition \ref{fact},
$|e'\cap e_0|\leq 1$.}

\red{
If $e'\cap e_0=\emptyset$,
by Lemma \ref{LE2.5},
then $e'\in (\hye/e_0)^*$.
Since $e'\subseteq \bigcup_{1\leq i\leq t}e_i$ and $e'\cap e_0=\emptyset$,
we have that $e'\cap P_0=\emptyset$,
implying that
$e'\subseteq \bigcup_{1\leq i\leq t}e_i^*$.}

\red{
If $|e'\cap e_0|=1$,
say $e'\cap e_0=\{u\}$,
then $\{v, w\}\in (\hye/e_0)^*$ by Lemma \ref{LE2.5}.
Since $e'\subseteq \bigcup_{1\leq i\leq t}e_i$ and $e'\cap e_0=\{u\}$,
we have that $P_0\neq \emptyset$ which implies
$\bigcup_{1\leq i\leq t}e_i^*=(\bigcup_{1\leq i\leq t}e_i-P_0)\cup \{w\}$.
Thus
$\{v, w\}\subseteq \bigcup_{1\leq i\leq t}e_i^*$.}

Hence the claim holds in this case.

\red{
\noindent
\textbf{Case 2:} $w\in \{u_1, \ldots, u_t\}$.}

\red{
Without loss of generality, suppose that $w=u_1$.
Since $w=u_1\in e_1^*\cap e_t^*$, 
by Lemma \ref{LE2.5}, 
we have $|e_1\cap e_0|=|e_t\cap e_0|=1$. By assumtion, $C^*$ is a minimal cycle in $({\hyh}/e_0)^*$, implying that $w\notin e_i^*$ for all $i=2,\ldots,t-1$,  and so $e_i\cap e_0=\emptyset$ for all $i=2,3,\ldots,t-1$ by Lemma \ref{LE2.5}. 
}

\begin{figure}[htp]
\centering
\labellist
\pinlabel {$e_1$} at 120 320
\pinlabel {$e_t$} at 230 385
\pinlabel {$x$} at 125 410
\pinlabel {$e_0$} at 90 465
\pinlabel {$e_2$} at 120 195
\pinlabel {$e_{t-1}$} at 345 320
\pinlabel {$u_2$} at 30 250
\pinlabel {$u_{t}$} at 350 400

\pinlabel {$e_1$} at 625 330
\pinlabel {$e_0$} at 715 415
\pinlabel {$e_t$} at 845 425
\pinlabel {$x$} at 610 400
\pinlabel {$y$} at 800 480
\pinlabel {$u_{t-1}$} at 940 400
\pinlabel {$u_2$} at 550 240
\pinlabel {$e_2$} at 650 190

\endlabellist

\includegraphics[scale=0.33]{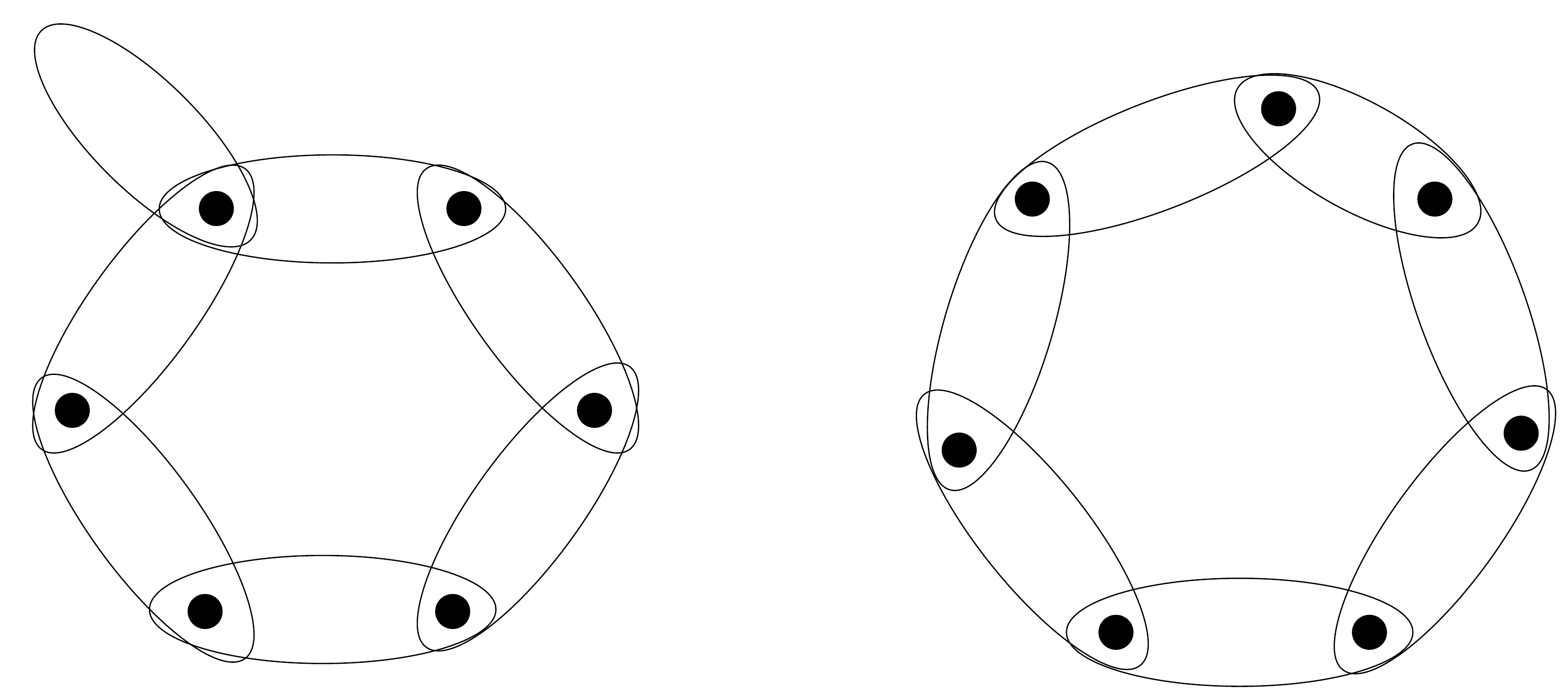}\\
\vspace*{0.3cm}
\scriptsize{\hspace*{0.2cm}(i)\hspace*{6.5cm} (ii)}
\caption{A cycle of $\hyh$ in Subcases 2.1 and 2.2}\label{np3}
\end{figure}

\red{
\textbf{Case 2.1:} $e_1\cap e_0=e_t\cap e_0=\{x\}$ for some $x\in V$.}

\red{In this subcase,
$e_1=(e_1^*\setminus \{w\})\cup \{x\}$,
$e_t=(e_t^*\setminus \{w\})\cup \{x\}$
and $e_i=e_i^*$ for all $i=2,3,\cdots,t-1$.
Thus 
$C=(x, e_1, u_2, e_2, u_3, \ldots, u_t, e_t, x)$ is a cycle in $\hyh$, as shown in Figure \ref{np3} (i).
Since $\hyh\in \hyl_0$,
there is an edge $e'=\{u, v\}\in \hye$ such that 
$e'\subseteq \bigcup_{1\leq i\leq t}e_i$.}

\red{If $x\notin \{u,v\}$, then $e'\in (\hye/e_0)^*$ 
and $e'\subseteq \bigcup_{1\leq i\leq t}e_i^*$.}

\red{If $x\in \{u,v\}$, say $x=u$, then $\{w,v\}\in (\hye/e_0)^*$ 
and $\{w,v\}\subseteq \bigcup_{1\leq i\leq t}e_i^*$.}

\red{Thus the claim holds for this subcase. 
}

\red{
\textbf{Case 2.2:} $e_1\cap e_0=\{x\}$ and $e_t\cap e_0=\{y\}$, 
where $x, y\in V$ and $x\neq y$.}

\red{
In this subcase,
$e_1=(e_1^*\setminus \{w\})\cup \{x\}$,
$e_t=(e_t^*\setminus \{w\})\cup \{y\}$
and $e_i=e_i^*$ for all $i=2,3,\cdots,t-1$.
Thus 
$C=\{y, e_0, x, e_1, u_2, e_2, u_3, e_3, \ldots, u_t, e_t, y\}$ is a cycle in $\hyh$, 
as shown in Figure \ref{np3} (ii).
Since $\hyh\in \hyl_0$,
there is an edge $e'=\{u, v\}\in \hye$ such that
\red{ 
$e'\subseteq \bigcup_{0\leq i\leq t}e_i$}. By Proposition \ref{fact},
$|e'\cap e_0|\leq 1$.}

\red{
If $e'\cap e_0=\emptyset$, then $e'\in (\hye/e_0)^*$ 
by Lemma \ref{LE2.5} and 
$e'\subseteq \big(\bigcup_{1\leq i\leq t}e_i\big)\setminus \{x, y\}$
implying that $e'\subseteq \bigcup_{1\leq i\leq t}e_i^*$.}

\red{
If $|e'\cap e_0|=1$,
say $e'\cap e_0=\{u\}$,
then $\{w, v\}\in (\hye/e_0)^*$ by Lemma \ref{LE2.5} and $v\in \big(\bigcup_{1\leq i\leq t}e_i\big)\setminus \{x,y\}$, implying that 
$\{w, v\}\in \bigcup_{1\leq i\leq t}e_i^*$.}

Thus the claim holds for this subcase. 
 
Therefore, the claim holds in Case 2 and the proof is complete.
\end{proof}

\section{Proving the main results}
\label{S4}

Now we are ready to prove Theorems \ref{Th-main1} and 
\ref{Th-main2}.

\begin{proof}[Proof of Theorem \ref{Th-main1}]
\red{Let $n$ be the number of vertices in $\hyh$ and $m$ be the number of edges in $\hyh$.
We shall prove the statement by induction on $m$.}

When $m=0$, $\hyh$ is an empty graph and $P(\hyh, \lambda)=\lambda^n$.
It is obvious that $(-1)^n P(\hyh, \lambda)>0$ for any $\lambda\in (-\infty, 0)$.

Assume that Theorem \ref{Th-main1} holds 
for all hypergraphs in $\hyl_0$
with  
less than $m$ $(m\geq 1)$ edges. 
Now let $\hyh=(V, \hye)$ be a hypergraph in $\hyl_0$,
where $|V|=n$ and $|\hye|=m$.
By (\ref{nEq1}), we can assume that $\hyh$ is Sperner.

\red{Let $e_0$ be the edge in $\hyh$ with the largest size.
If $|e_0|=2$, then $\hyh$ is a graph and Theorem \ref{Th-main1} has been 
proved that it holds for a graph
(see \cite{fengming2005chromatic, jackson1993zero, 
read1968introduction, read1988chromatic}).}

\red{If $|e_0|>2$,}
by Theorem \ref{TH2.3} and (\ref{nEq1}), we have that 
\begin{equation}\label{nEq6}
P(\hyh, \lambda)=P(\hyh-e_0, \lambda)-P(\hyh/e_0, \lambda)=P(\hyh-e_0, \lambda)-P((\hyh/e_0)^*, \lambda).
\end{equation}
As $\hyh$ is a Sperner hypergraph in $\hyl_0$,
by Lemma \ref{LE2.6},
both $\hyh-e_0$ and $(\hyh/e_0)^*$ belong to $\hyl_0$.
Thus, by the inductive assumption, for any real $\lambda\in (-\infty, 0)$,
\begin{equation}\label{Eqn7}
(-1)^n P(\hyh-e_0, \lambda)>0 \text{~~and~~} (-1)^{n-|e_0|+1}P((\hyh/e_0)^*, \lambda)>0.
\end{equation}
As $\hyh\in \hyl_0$, $|e_0|$ is even.
Thus,
from (\ref{nEq6}) and (\ref{Eqn7}), for any real $\lambda\in (-\infty, 0)$,
\begin{align*}
(-1)^n P(\hyh, \lambda)= & (-1)^n (P(\hyh-e_0, \lambda)-P((\hyh/e_0)^*, \lambda))\\
= & \underbrace{(-1)^n   P(\hyh-e_0, \lambda)}_{>0} + \underbrace{(-1)^{|e_0|}}_{>0} 
\underbrace{(-1)^{n-|e_0|+1}P((\hyh/e_0)^*, \lambda)}_{>0}\\
> & 0.
\end{align*}
The proof is complete.
\end{proof}


\red{For any hypergraph $\hyh$,
let 
\begin{equation}\label{nEq8}
Q(\hyh, \lambda)=\frac{1}{\lambda}\cdot P(\hyh, \lambda).
\end{equation}
Clearly, $Q(\hyh, \lambda)$ is also 
a polynomial in $\lambda$,
as $P(\hyh, \lambda)$ is a polynomial in $\lambda$ and $P(\hyh, 0)=0$ (see \cite{tomescu1998chromatic}).}
To prove Theorem \ref{Th-main2}, 
it suffices to establish the result below.

\begin{theo}\label{th4-1}
\red{If $\hyh=(V, \hye)$ is a hypergraph in 
$\hyl_0'$,} 
then for any real $\lambda\in [0, 1)$,
\begin{equation}\label{nEq90}
(-1)^{|V|+1}Q(\hyh, \lambda)>0.
\end{equation}
\end{theo}

\begin{proof}
Let $\hye'(\hyh)=\{e\in \hye: |e|>2\}$ and $m'=|\hye'(\hyh)|$.
We shall prove the statement by induction on $m'$.
When $m'=0$, $\hyh\in \hyl'_0$ implies that 
$\hyh$ is a connected graph.
Thus (\ref{nEq90}) holds in this case
(see \cite{fengming2005chromatic, jackson1993zero}).

Assume that (\ref{nEq90}) holds
for all hypergraphs $\hyh$ in $\hyl'_0$
with  $|\hye'(\hyh)|<m'$ ($m'\geq 1$).

Now let $\hyh=(V,\hye)$ be a Sperner hypergraph in $\hyl'_0$
with $|\hye'(\hyh)|=m'$.
Let $e_0\in \hye'(\hyh)$. 
By Theorem \ref{TH2.3} and (\ref{nEq1}), 
\begin{equation}\label{nEq9}  
P(\hyh, \lambda)=P(\hyh-e_0, \lambda)-P(\hyh/e_0, \lambda)
=P(\hyh-e_0, \lambda)-P((\hyh/e_0)^*, \lambda).
\end{equation}
By (\ref{nEq8}) and (\ref{nEq9}), 
\begin{equation}\label{nEq11}
Q(\hyh, \lambda)=Q(\hyh-e_0, \lambda)-Q((\hyh/e_0)^*, \lambda).
\end{equation}

As $|e_0|>2$ and $\hyh\in \hyl'_0$, 
 $\hyh-e_0\in \hyl'_0$ holds. 
 
 \red{By Lemma \ref{LE2.6},  
$(\hyh/e_0)^*\in \hyl_0$. 
 Let $\hyh^2$ be the hypergraph $(V, \hye_2)$.
As $\hyh\in \hyl'_0$,  
$\hyh^2$ is connected.
By Proposition \ref{fact}, }
$|e'\cap e_0|\leq 1$ for each $e'\in \hye_2$.
Thus $\hyh^2/e_0$ is connected, 
implying that $(\hyh/e_0)^*\in \hyl'_0$.

Since $|e_0|>2$, 
$|\hye'(\hyh-e_0)|<m'$ and $|\hye'((\hyh/e_0)^*)|<m'$.
Thus, by inductive assumption, for any real $\lambda\in [0, 1)$, 
\begin{equation}\label{nEq12}
(-1)^{|V|+1}Q(\hyh-e_0, \lambda)>0 \text{~~and~~} 
(-1)^{|V|-|e_0|+2} Q((\hyh/e_0)^*, \lambda)>0.
\end{equation}

As $\hyh\in \hyl_0$, $|e_0|$ is even.
Thus by (\ref{nEq11}) and (\ref{nEq12}), for any real $\lambda\in [0, 1)$, 
\begin{align*}
(-1)^{|V|+1} Q(\hyh, \lambda)= & (-1)^{|V|+1} (Q(\hyh-e_0, \lambda)-Q((\hyh/e_0)^*, \lambda))\\
= & \underbrace{(-1)^{|V|+1} Q(\hyh-e_0, \lambda)}_{>0}
+\underbrace{(-1)^{|e_0|}}_{>0}\underbrace{(-1)^{|V|-|e_0|+2}Q((\hyh/e_0)^*, \lambda)}_{>0}\\
> & 0.
\end{align*}
The proof is complete.
\end{proof}

\section{Further study}
\label{S5}

We are not sure if Theorems~\ref{Th-main1} and~\ref{Th-main2}
can be extended to larger families of 
hypergraphs. Especially for Theorem~\ref{Th-main2},
we are not sure if its conclusion holds 
for some connected hypergraph $\hyh=(V,\hye)$ whose 
spanning subhypergraph 
$(V,\hye_2)$ is not connected. 

\begin{prob}\label{p5-1}
Find a sufficient condition weaker than that in 
Theorem~\ref{Th-main2} for a hypergraph $\hyh=(V,\hye)$
such that 
$P(\hyh,\lambda)\ne 0$ for all $\lambda\in (0,1)$.
\end{prob}

We wonder if Theorems~\ref{Th-main1} and~\ref{Th-main2}
can be extended to the chromatic polynomials of 
mixed hypergraphs.
A \textit{mixed hypergraph} is a triple $\hyh=(V, \hyc, \hyd)$, 
where $V$ is a finite set, called 
the \textit{vertex set} of $\hyh$, 
and $\hyc$ and $\hyd$
are collections of subsets of $V$, 
called the \textit{$\hyc$-edge set} and \textit{$\hyd$-edge set} respectively.
A \textit{proper $k$-coloring} of  $\hyh=(V, \hyc, \hyd)$ is a mapping \red{$\phi:V\rightarrow \{1,2,\ldots,k\}$}
such that $|\{\phi(v): v\in e\}|<|e|$ 
holds for each $e\in \hyc$
and $|\{\phi(v): v\in e\}|\ge 2$ 
holds for each $e\in \hyd$.
The coloring of mixed hypergraphs 
was introduced by Voloshin \cite{voloshin1995upper} in 1995.
Since then,
 the study of colorings of 
mixed hypergraphs has developed into 
a separate topic and it has been applied in many fields, 
such as database management, channel assignments, cyber security \cite{hirotsugu2017access, jaffe2012price, kral2007mixed, li2016towards}.

The \textit{chromatic polynomial} of a mixed hypergraph 
$\hyh=(V, \hyc, \hyd)$ 
is the function, denoted by $P_{mix}(\hyh, \lambda)$,
which counts the number of proper $\lambda$-colorings of 
$\hyh$ whenever $\lambda$ is a positive integer.
Evidently, 
a weak proper $k$-coloring of a hypergraph $\hyh=(V, \hye)$ 
is actually a proper $k$-coloring of the mixed hypergraph $\hyh'=(V, \emptyset, \hye)$,
implying that $P(\hyh,\lambda)=P_{mix}(\hyh',\lambda)$.
Thus, the chromatic polynomials of 
mixed hypergraphs extend the chromatic polynomials
of hypergraphs.

\begin{prob}\label{p5-2}
Extend Theorems~\ref{Th-main1} and~\ref{Th-main2}
to mixed hypergraphs.
\end{prob}

\vspace{0.3 cm}

\noindent {\bf Acknowledgment}: 
The authors wish to thank the referees 
for their very helpful comments and suggestions.
The research was partially supported
by NTU AcRF project (RP 3/16 DFM) of Singapore.

\bibliography{Refer}{} 

\begin{thebibliography}{10}

\bibitem{agnarsson2004strong}
G.~Agnarsson and M.~M. Halld{\'o}rsson.
\newblock Strong colorings of hypergraphs.
\newblock In {\em International Workshop on Approximation and Online
  Algorithms}, pages 253--266. Springer, 2004.

\bibitem{allagan2007chromatic}
J.~A. Allagan.
\newblock The chromatic polynomials of some linear uniform hypergraphs.
\newblock {\em Congr. Numer.}, \textbf{187}:156--160, 2007.

\bibitem{allagan2014chromaticwheel}
J.~A. Allagan.
\newblock Chromatic polynomials of some $(m, l)$-hyperwheels.
\newblock {\em Comput. Sci. J. Moldova}, \textbf{22}:21--36, 2014.

\bibitem{allagan2014chromatic}
J.~A. Allagan and D.~Slutzky.
\newblock Chromatic polynomials of some mixed hypergraphs.
\newblock {\em Australas. J. Combin.}, \textbf{58}:197--213, 2014.

\bibitem{biggs1993algebraic}
N.~L. Biggs.
\newblock {\em Algebraic {G}raph {T}heory}.
\newblock Cambridge University Press, 1993.

\bibitem{birkhoff1946chromatic}
G.~D. Birkhoff and D.~C. Lewis.
\newblock Chromatic polynomials.
\newblock {\em Trans. Amer. Math. Soc.}, \textbf{60}:355--451, 1946.

\bibitem{borowiecki2000chromatic}
M.~Borowiecki and E.~{\L}azuka.
\newblock Chromatic polynomials of hypergraphs.
\newblock {\em Discuss. Math. Graph Theory}, \textbf{20}:293--301, 2000.

\bibitem{dohmen1995broken}
K.~Dohmen.
\newblock A broken-circuits-theorem for hypergraphs.
\newblock {\em Arch. Math.}, \textbf{64}:159--162, 1995.

\bibitem{dong2011zero}
F.~Dong and B.~Jackson.
\newblock A zero-free interval for chromatic polynomials of nearly 3-connected
  plane graphs.
\newblock {\em SIAM J. Discrete Math.}, \textbf{25}:1103--1118, 2011.

\bibitem{fengming2005chromatic}
F.~Dong, K.~M. Koh, and K.~L. Teo.
\newblock {\em Chromatic {P}olynomials and {C}hromaticity of {G}raphs}.
\newblock World Scientific, 2005.

\bibitem{erdHos1966chromatic}
P.~Erd\H{o}s and A.~Hajnal.
\newblock On chromatic number of graphs and set-systems.
\newblock {\em Acta Math. Hungar.}, \textbf{17}:61--99, 1966.

\bibitem{hirotsugu2017access}
K.~Hirotsugu and M.~Tetsuya.
\newblock Access control model for the inference attacks with access histories.
\newblock In {\em 2017 IEEE 41st Annual Computer Software and Applications
  Conference (COMPSAC)}, pages 498--503. IEEE, 2017.

\bibitem{jackson1993zero}
B.~Jackson.
\newblock A zero-free interval for chromatic polynomials of graphs.
\newblock {\em Combin. Probab. Comput.}, \textbf{2}:325--336, 1993.

\bibitem{jackson2003zeros}
B.~Jackson.
\newblock Zeros of chromatic and flow polynomials of graphs.
\newblock {\em J. Geom.}, \textbf{76}:95--109, 2003.

\bibitem{jaffe2012price}
A.~Jaffe, T.~Moscibroda, and S.~Sen.
\newblock On the price of equivocation in byzantine agreement.
\newblock In {\em Proceedings of the 2012 ACM {S}ymposium on Principles of
  {D}istributed {C}omputing}, pages 309--318. ACM, 2012.

\bibitem{jones1976colourings}
R.~P. Jones.
\newblock {\em Colourings of hypergraphs}.
\newblock PhD thesis, University of London, 1976.

\bibitem{jones1976some}
R.~P. Jones.
\newblock Some results of chromatic hypergraph theory proved by ``reduction to
  graphs''.
\newblock In {\em Colloque CNRS, ProblmesCombinatories et Thorie des Graphes},
  volume \textbf{260}, pages 249--250, 1976.

\bibitem{kral2007mixed}
D.~Kr$\acute{\text{a}}$l'.
\newblock Mixed hypergraphs and other coloring problems.
\newblock {\em Discrete Math.}, \textbf{307}:923--938, 2007.

\bibitem{li2016towards}
C.~Li, M.~Hurfin, Y.~Wang, and L.~Yu.
\newblock Towards a restrained use of non-equivocation for achieving iterative
  approximate byzantine consensus.
\newblock In {\em 2016 IEEE International Parallel and Distributed Processing
  Symposium (IPDPS)}, pages 710--719. IEEE, 2016.

\bibitem{read1968introduction}
R.~C. Read.
\newblock An introduction to chromatic polynomials.
\newblock {\em J. Combin. Theory}, \textbf{4}:52--71, 1968.

\bibitem{read1988chromatic}
R.~C. Read and W.~T. Tutte.
\newblock Chromatic polynomials.
\newblock In {\em Selected {T}opics in {G}raph {T}heory 3}, pages 15--42.
  Academic Press, 1988.

\bibitem{royle2009recent}
G.~Royle.
\newblock Recent results on chromatic and flow roots of graphs and matroids.
\newblock In {\em Surveys in Combinatorics 2009, S. Huczynska, JD Mitchell, CM
  Roney-Dougal eds}, pages 289--327. Cambridge University Press, Cambridge,
  2009.

\bibitem{sokal2004chromatic}
A.~D. Sokal.
\newblock Chromatic roots are dense in the whole complex plane.
\newblock {\em Combin. Probab. Comput.}, \textbf{13}:221--261, 2004.

\bibitem{thomassen1997zero}
C.~Thomassen.
\newblock The zero-free intervals for chromatic polynomials of graphs.
\newblock {\em Combin. Probab. Comput.}, \textbf{6}:497--506, 1997.

\bibitem{thomassen2000chromatic}
C.~Thomassen.
\newblock Chromatic roots and hamiltonian paths.
\newblock {\em J. Combin. Theory Ser. B}, \textbf{80}:218--224, 2000.

\bibitem{tomescu1998chromatic}
I.~Tomescu.
\newblock Chromatic coefficients of linear uniform hypergraphs.
\newblock {\em J. Combin. Theory Ser. B}, \textbf{72}:229--235, 1998.

\bibitem{tomescu2004sunflower}
I.~Tomescu.
\newblock Sunflower hypergraphs are chromatically unique.
\newblock {\em Discrete Math.}, \textbf{285}:355--357, 2004.

\bibitem{tomescu2007chromaticity}
I.~Tomescu.
\newblock On the chromaticity of sunflower hypergraphs {SH}(n, p, h).
\newblock {\em Discrete Math.}, \textbf{307}:781--786, 2007.

\bibitem{tomescu2014hypergraphs}
I.~Tomescu.
\newblock Hypergraphs with pendant paths are not chromatically unique.
\newblock {\em Discuss. Math. Graph Theory}, \textbf{34}:23--29, 2014.

\bibitem{tomescu2009some}
I.~Tomescu and S.~A. Bokhary.
\newblock Some properties of chromatic coefficients of linear uniform
  hypergraphs.
\newblock {\em Graphs Combin.}, \textbf{25}:639, 2009.

\bibitem{tutte1974chromials}
W.~T. Tutte.
\newblock Chromials.
\newblock In {\em Hypergraph Seminar}, pages 243--266. Springer, 1974.

\bibitem{voloshin1993mixed}
V.~I. Voloshin.
\newblock The mixed hypergraphs.
\newblock {\em Comput. Sci. J. Moldova}, \textbf{1}:1, 1993.

\bibitem{voloshin1995upper}
V.~I. Voloshin.
\newblock On the upper chromatic number of a hypergraph.
\newblock {\em Australas. J. Combin.}, \textbf{11}:25--45, 1995.

\bibitem{voloshin2002coloring}
V.~I. Voloshin.
\newblock {\em Coloring {M}ixed {H}ypergraphs: {T}heory, {A}lgorithms, and
  {A}pplications}.
\newblock Amer. Math. Soc., 2002.

\bibitem{walter2009some}
M.~Walter.
\newblock Some results on chromatic polynomials of hypergraphs.
\newblock {\em Electron. J. Combin.}, \textbf{16}:R94, 2009.

\bibitem{whitney1932logical}
H.~Whitney.
\newblock A logical expansion in mathematics.
\newblock {\em Bull. Amer. Math. Soc.}, \textbf{38}:572--579, 1932.

\bibitem{woodall1977zeros}
D.~R. Woodall.
\newblock Zeros of chromatic polynomials.
\newblock In {\em Combinatorial Surveys: Proceedings of the Sixth British
  Combinatorial Conference}, pages 199--223, 1977.

\bibitem{woodall1997largest}
D.~R. Woodall.
\newblock The largest real zero of the chromatic polynomial.
\newblock {\em Discrete Math.}, \textbf{172}:141--153, 1997.

\bibitem{zhang2017properties}
R.~Zhang and F.~Dong.
\newblock Properties of chromatic polynomials of hypergraphs not held for
  chromatic polynomials of graphs.
\newblock {\em European J. Combin.}, \textbf{64}:138--151, 2017.

\end{thebibliography}
\bibliographystyle{abbrv}

\end{document}